\documentclass[11pt]{amsart}

\usepackage{amssymb,epstopdf,subcaption,mathtools,tikz-cd,pb-diagram,thmtools, thm-restate,pinlabel,pifont,enumitem,stmaryrd}
\usepackage{graphicx}

\usepackage[all]{xy}

\newcommand{\nc}{\newcommand}
\nc{\dmo}{\DeclareMathOperator}
\nc{\nt}{\newtheorem}

\newtheorem{theorem}{Theorem}[section]
\newtheorem*{mainthm}{Main Theorem}

\newtheorem{lemma}[theorem]{Lemma}
\newtheorem{proposition}[theorem]{Proposition}
\newtheorem{corollary}[theorem]{Corollary}
\theoremstyle{definition}

\newtheorem{question}[theorem]{Question}

\newtheorem{remark}[theorem]{Remark}
\theoremstyle{remark}

\nc{\cut}{\!\ssearrow\!}

\dmo{\Diff}{Diff}
\dmo{\Mod}{Mod}
\dmo{\SMod}{SMod}
\dmo{\I}{\mathcal{I}}
\dmo{\SO}{SO}
\dmo{\Orth}{O}
\dmo{\Sp}{Sp}
\dmo{\SL}{SL}
\dmo{\GL}{GL}
\dmo{\im}{im}
\dmo{\Emb}{Emb}
\dmo{\PSp}{PSp}
\dmo{\PSL}{PSL}
\dmo{\Homeo}{Homeo}
\dmo{\Twist}{Twist}
\dmo{\Aut}{Aut}
\dmo{\Nil}{Nil}
\dmo{\Sol}{Sol}
\dmo{\Isom}{Isom}
\dmo{\Out}{Out}
\dmo{\orb}{orb}

\nc{\Z}{\mathbb Z}
\nc{\N}{\mathcal N}
\nc{\R}{\mathbb R}
\nc{\F}{\mathcal F}
\nc{\C}{\mathbb{C}}

\nc{\ga}{\gamma}
\nc{\de}{\delta}
\nc{\ep}{\epsilon}

\nc{\flm}{\lambda_{2}}

\nc{\normalclosure}[1]{\ensuremath{\left \langle \left \langle #1 \right \rangle \right \rangle}}

\nc{\margin}[1]{\marginpar{\scriptsize #1}}

\nc{\p}[1]{\bigskip\noindent\textbf{#1.}}
\nc{\lei}[1]{{\color{red} \sf  L: [#1]}}
\nc{\bena}[1]{{\color{blue} \sf  B: [#1]}}

\title{Nielsen Realization for sphere twists on 3-manifolds}
\author{Lei Chen}
\author{Bena Tshishiku}

\address{Lei Chen \\ Department of Mathematics\\ University of Maryland \\ 4176 Campus Drive\\ College Park, MD 20742 \\  chenlei@umd.edu}

\address{Bena Tshishiku \\ Department of Mathematics\\ Brown University \\ 151 Thayer St.  \\ Providence, RI, 02912\\  bena$\_$tshishiku@brown.edu.   }

\begin{document}

\maketitle

\vspace*{-4ex}

\begin{abstract}
For a 3-manifold $M$, the twist group $\Twist(M)$ is the subgroup of the mapping class group $\Mod(M)$ generated by twists about embedded $2$-spheres. We study the Nielsen realization problem for subgroups of $\Twist(M)$. We prove that a nontrivial subgroup $G<\Twist(M)$ is realized by diffeomorphisms if and only if $G$ is cyclic and $M$ is a connected sum of lens spaces, including $S^1\times S^2$. We also apply our methods to the Burnside problem for 3-manifolds and show that $\Diff(M)$ does not contain an infinite torsion group when $M$ is reducible and not a connected sum of lens spaces.
\end{abstract}



\section{Introduction}

The mapping class group $\Mod(M)$ of a smooth, closed oriented manifold $M$, is the group of isotopy classes of orientation-preserving diffeomorphisms of $M$. Denoting $\Diff^+(M)$ the group of orientation-preserving diffeomorphisms, there is a natural projection map \[\pi:\Diff^+(M)\to\Mod(M)\]
sending a diffeomorphism to its isotopy class. We say that a subgroup $i: G\hookrightarrow \Mod(M)$ is \emph{realizable} if there is a homomorphism of $\rho: G\to \Diff^+(M)$ such that $\pi \circ \rho = i$.
\[\begin{xy}
(-30,0)*+{G}="A";
(0,15)*+{\Diff^+(M)}="B";
(0,0)*+{\Mod(M)}="C";
{\ar@{-->}"A";"B"}?*!/_3mm/{\rho};
{\ar "B";"C"}?*!/_3mm/{\pi};
{\ar "A";"C"}?*!/^3mm/{i};
\end{xy}\]
The Nielsen realization problem asks which finite groups $G<\Mod(M)$ are realizable.  When $M$ is a surface, every finite subgroup of $\Mod(M)$ is realizable by work of Kerkchoff \cite{kerckhoff}. For other manifolds, only sporadic results have been obtained; see  \cite[\S2]{pardon} for a summary of results in dimension 3, \cite{farb-looijenga,Konno2, lee, Konno} for results in dimension 4, and \cite{farrell-jones,block-weinberger,bustamante-tshishiku} for higher dimensions. 

When $M$ is a 3-manifold, the \emph{twist group} $\Twist(M)<\Mod(M)$ is the subgroup generated \emph{sphere twists} (defined in \S\ref{sec:sphere-twists}). McCullough \cite{mccullough} proved that $\Twist(M)\cong(\Z/2\Z)^d$ for some $d\ge0$. We address the Nielsen realization problem for subgroups of $\Twist(M)$. This problem was studied by Zimmermann \cite{zimmermann} for $M=\#_k(S^1\times S^2)$, but there is an error in his argument; see \cite{zimmermann-erratum}. Nevertheless, using some of the same ideas, we correct the argument, and we generalize from $\#_k(S^1\times S^2)$ to all $3$-manifolds, giving a precise condition for which subgroups can be realized or not.

To state the main result, recall that for any pair of coprime integers $p,q$, there is a lens space $L(p,q)$. Every lens space is covered by $S^3$ with the exception of $L(0,1)\cong S^1\times S^2$.

\begin{mainthm}
Fix a closed, oriented $3$-manifold $M$, and fix a nontrivial subgroup $1\neq G<\Twist(M)$. Then $G$ is realizable if and only if $G$ is cyclic and $M$ is diffeomorphic to a connected sum of lens spaces. 
\end{mainthm}

\p{Application to the Burnside problem for diffeomorphism groups} 
Recall that a \emph{torsion group} is a group where every element has finite order. The existence of finitely-generated, infinite torsion groups, known as \emph{Burnside groups}, was proved by Golod--Shafarevich \cite{Golod} \cite{GS} and Adian-Novikov \cite{AN}. 
E. Ghys and B. Farb asked if the homeomorphism group of a compact manifold can contain a Burnside group; see \cite[Question 13.2]{Fisher2} and \cite[\S5]{Fisher1}. 
As an application of the tools used to prove the Main Theorem, we prove the following result on Burnside problem for $\Diff(M)$.
\begin{theorem}\label{Burnside}
Let $M$ be a compact oriented $3$-manifold. Assume that $M$ is reducible and not a connected sum of lens spaces. Then $\Diff(M)$ does not contain a Burnside group.
\end{theorem}

\p{Smooth vs.\ topological Nielsen realization} The topological mapping class group $\Mod_H(M)$ is defined as the group of isotopy classes of orientation-preserving homeomorphisms of $M$. There is a natural projection map 
\[\Homeo^+(M)\to \Mod_H(M),\]
and the Nielsen realization problem can also be asked for subgroups of $\Mod_H(M)$. For 3-manifolds, the smooth and topological mapping class groups coincide  $\Mod(M)\cong\Mod_H(M)$ by Cerf \cite{cerf}, who proved that $\Diff^+(M)$ and $\Homeo^+(M)$ are homotopy equivalent. Surprisingly, the realization problem for finite groups is also the same in the topological and smooth categories in dimension 3. For dimension $4$, even the connected components are different by work of Ruberman \cite{Ruberman}. The work of Baraglia--Konno \cite{Konno2} gives a mapping class that can be realized as homeomorphism but not as diffeomorphism and Konno \cite{Konno} generalizes the results to more manifolds.

\begin{theorem}[Pardon, Kirby--Edwards]\label{thm:pardon}
Let $M$ be a closed oriented 3-manifold. A finite subgroup $G<\Mod(M)$ is realizable by homeomorphisms if and only if it is realizable by diffeomorphisms.
\end{theorem}

In particular, this allows us to strengthen the conclusion of the Main Theorem. We emphasize that the group $G$ in Theorem \ref{thm:pardon} is finite; however, we do not know an example of a 3-manifold $M$ and an infinite group $G<\Mod(M)$ that is realizable by homeomorphisms but not diffeomorphisms.

\begin{proof}[Proof of Theorem \ref{thm:pardon}]
Let $\rho:G\to\Homeo^+(M)$ be a realization of $G<\Mod(M)$. By Pardon \cite{pardon}, $\rho$ can be approximated uniformly by a smooth action $\rho': G\to\Diff^+(M)$. Since $\Homeo^+(M)$ is locally path-connected by Kirby--Edwards \cite{kirby-edwards}, we know that $\rho'(g)$ and $\rho(g)$ are isotopic in $\Homeo^+(M)$ for each $g\in G$, and hence also isotopic in $\Diff^+(M)$ since $\Mod_H(M)=\Mod(M)$. 
\end{proof}

\p{Related work} The following remarks connect the Main Theorem to other previous work. 

\begin{remark}[Twist group for $S^1\times S^2$] 
The group $\Twist(S^1\times S^2)$ is isomorphic to $\Z/2\Z$. We construct a realization of this group in \S\ref{sec:construction}. This example seems to be overlooked in some of the literature on finite group actions on geometric 3-manifolds. It is a folklore conjecture of Thurston that any finite group action on a geometric 3-manifold is geometric (i.e.\ acts isometrically on some geometric structure). It is easy to see that our realization of the twist group, which also appears in work of Tollefson \cite{tollefson_involutions}, does not preserve any geometric structure on $S^1\times S^2$, so it is a simple counterexample to Thurston's conjecture.

According to Meeks--Scott \cite{meeks-scott}, Thurston proved some cases of his conjecture, but these results were not published. Meeks--Scott \cite{meeks-scott} proved Thurston's conjecture for manifolds modeled on $\mathbb H^2\times\R$, $\widetilde{\SL_2(\R)}$, $\Nil$, $\mathbb E^3$, and $\Sol$. In \cite[Thm.\ 8.4]{meeks-scott} it is asserted (incorrectly) that Thurston's conjecture also holds for 3-manifolds modeled on $S^2\times\R$ (in particular $S^2\times S^1$); they give an argument, but in the case when some $g\in G$ has positive-dimensional fixed set (as is the case for our realization of $\Twist(S^1\times S^2)$), they cite a preprint of Thurston that seems to have never appeared.
\end{remark}

\begin{remark}[Sphere twists in dimension 4]
For a 4-manifold $W$, for each embedded 2-sphere $S\subset W$ with self-intersection $S\cdot S=-2$, there is a sphere twists $\tau_S\in\Mod(W)$, which has order 2. There are several results known about realizing the subgroup generated by a sphere twist, both positive and negative; see Farb--Looijenga \cite[Cor.\ 1.10]{farb-looijenga}, Konno \cite[Thm.\ 1.1]{Konno}, and Lee \cite[Rmk.\ 1.7]{lee}. It would be interesting to determine precisely when a sphere twist is realizable in dimension 4. \end{remark}

\p{About the proof of the Main Theorem} The proof is divided into two parts: construction and obstruction (corresponding to the ``if" and ``only if" directions in the theorem statement). For the obstruction part of the argument, we prove the following constraint on group actions on reducible 3-manifolds. 

\begin{theorem}\label{thm:contraint}
Let $M$ be a closed, oriented, reducible $3$-manifold. Let $G<\Diff^+(M)$ be a finite subgroup that acts trivially on $\pi_1(M)$. Then $G$ is cyclic. If $G$ is nontrivial, then $M$ is a connected sum of lens spaces. 
\end{theorem}

\p{Section outline} In \S\ref{sec:sphere-twists}, we recall results about sphere twists and the twist group. In \S\ref{sec:invariant-spheres} we explain results from minimal surface theory that allow us to decompose a given action into actions on irreducible 3-manifolds. In \S\ref{sec:obstruction} and \ref{sec:construction} we prove the ``obstruction" and ``construction" parts of the Main Theorem, respectively. In \S\ref{Burn}, we prove Theorem \ref{Burnside}.

\p{Acknowledgement} The authors are supported by NSF grants DMS-2203178, DMS-2104346 and DMS-2005409. This work is also supported by NSF Grant No. DMS-1439786 while the first author visited ICERM in Spring 2022.

\section{Sphere twists and the twist subgroup} \label{sec:sphere-twists}

In this section we collect some facts about sphere twists and the group they generate. In \S\ref{sec:individual-twist} we recall the definition of sphere twists and recall that they act trivially on $\pi_1(M)$ and $\pi_2(M)$. In \S\ref{sec:twist-group} we give a computation for the twist group of any closed, oriented 3-manifold (Theorem \ref{thm:twist-group}). The computation can be deduced by combining different results from the literature and gives a precise generating set. 

\subsection{Sphere twists and their action on homotopy groups} \label{sec:individual-twist}

Fix a closed oriented $3$-manifold $M$. We recall the definition of a sphere twist. Fix an embedded $2$-sphere $S\subset M$ with a tubular neighborhood $U\cong S\times [0,1]\subset M$, and fix a closed path $\phi:[0,1]\to\SO(3)$ based at the identity that generates $\pi_1(\SO(3))$. Define a diffeomorphism of $U$ by 
\begin{equation}\label{eqn:twist}
T_S(x,t)=\big(\phi(t)(x),t\big)
\end{equation}
and extend by the identity to obtain a diffeomorphism $T_S$ of $M$. The isotopy class $\tau_S\in\Mod(M)$ of $T_S$ is called a \emph{sphere twist}. The \emph{twist subgroup} of $\Mod(M)$, denoted $\Twist(M)$, is the subgroup generated by all sphere twists. 

\begin{lemma}[Action of sphere twists on homotopy groups]\label{lem:twist-action}
Let $M$ be a closed, oriented $3$-manifold with a $2$-sided embedded sphere $S\subset M$. Then $\tau_S$ acts trivially on $\pi_1(M)$ and $\pi_2(M)$. 
\end{lemma}

\begin{remark}[Action on $\pi_1(M)$ vs.\ $\pi_1(M,*)$]
When we refer to the action of a diffeomorphism $f\in\Diff(M)$ on $\pi_1(M)$, we mean as an outer automorphism. Technically, this is not an action, but there is a well-defined homomorphism $\Diff(M)\to\Out(\pi_1(M))$. When $f$ has a fixed point $*\in M$, the action of $f$ on $\pi_1(M,*)$ refers to the induced automorphisms $f_*:\pi_1(M,*)\to\pi_1(M,*)$. This distinction will be important in later sections. If $f$  acts trivially on $\pi_1(M)$ it can be isotoped to $f'\in\Diff(M,*)$ that acts on $\pi_1(M,*)$ by conjugation (generally nontrivial). 
\end{remark}

Lemma \ref{lem:twist-action} is well-known. That sphere twists act trivially on $\pi_1(M)$ is implicit in \cite{mccullough}. This fact may be proved as follows. Let $*\in M$ be a fixed point of the diffeomorphism $T_S$ defined in (\ref{eqn:twist}). After choosing a prime decomposition of $M$, one can show that each element of $\pi_1(M,*)$ is represented by a loop contained entirely in the fixed set of $T_S$. That sphere twists act trivially on $\pi_2(M)$ follows from \cite[Lem.\ 1.1]{mccullough}. It can also be deduced from a general result of Laudenbach that says that if $f\in\Diff(M,*)$ acts trivially on $\pi_1(M,*)$, then it also acts trivially on $\pi_2(M,*)$; see \cite[Thm.\ 2.4]{BBP}.


\subsection{Generators and relations in the twist group} \label{sec:twist-group}

In this section we compute $\Twist(M)$ for every closed, oriented 3-manifold. 

\begin{theorem}\label{thm:twist-group}
Let $M$ be a closed, orientable 3-manifold with prime decomposition $M=\#_k(S^1\times S^2)\#P_1\#\cdots\#P_\ell$, where the $P_i$ are irreducible. Let $\ell'\le\ell$ be the number of the $P_i$ that are lens spaces. Then 
\[\Twist(M)\cong\begin{cases}
(\Z/2\Z)^k&\text{ if } \ell=\ell'\\
(\Z/2\Z)^{k+\ell-\ell'-1}&\text{ if } \ell>\ell'.
\end{cases}\]
\end{theorem}

We were not able to find Theorem \ref{thm:twist-group} in the literature, although it can be deduced by combining various old results.

For the proof of Theorem \ref{thm:twist-group}, we will use the following explicit construction of $M$. Let $X$ be the complement of $\ell+2k$ open disks in $S^3$. For each $P_i$ choose a closed embedded disk $D_i\subset P_i$, and let $Y$ be the compact manifold 
\begin{equation}\label{eqn:prime-factors}Y:=\left[\coprod_{i=1}^\ell P_i\setminus \text{int}(D_i)\right]\sqcup\left[\coprod_{k}S^2\times[-1,1]\right].\end{equation}
We form $M$ by gluing $X$ and $Y$ along their boundary $\partial X\cong\partial Y$. See Figure \ref{fig:sphere-twist}. 

\begin{figure}[h!]
\labellist
\pinlabel $X$ at 100 80
\pinlabel $P_1$ at 10 40
\pinlabel $P_2$ at 115 10
\pinlabel $P_3$ at 190 50
\pinlabel $[-1,1]\times S^2$ at -20 110
\pinlabel $[-1,1]\times S^2$ at 215 110
\endlabellist
\centering
\includegraphics[scale=.6]{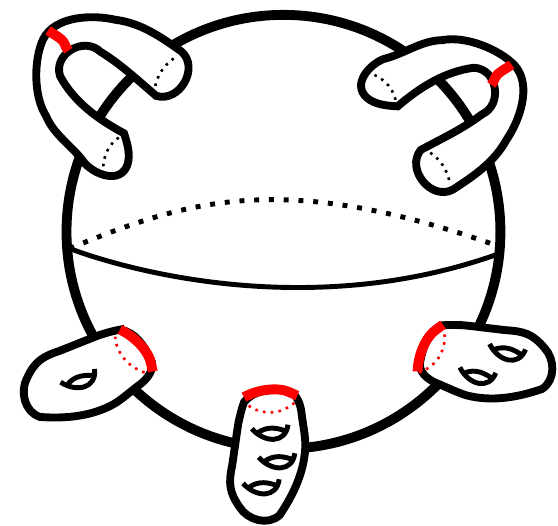}
\caption{Construction of $M=\#_k(S^1\times S^2)\# P_1\#\cdots\# P_\ell$. The red set represents spheres that generate $\Twist(M)$.}
\label{fig:sphere-twist}
\end{figure}

\begin{lemma}[Pants relation]\label{lem:pants}
Let $Z$ be the complement in $S^3$ of three disjoint $3$-balls. Let $S_1,S_2,S_3$ be the three boundary components of $Z$. Then $\tau_{S_1} \tau_{S_2} \tau_{S_3}=1$ in $\Mod(Z)$. 
\end{lemma}

Lemma \ref{lem:pants} can be proved in an elementary fashion by constructing an explicit isotopy. We explain the idea briefly. Fix a properly embedded arc $\alpha\subset Z$ joining the spheres $S_1$ and $S_2$. Let $N$ be a regular neighborhood of $S_1\cup\alpha\cup S_2$. Observe that $N$ is diffeomorphic to $Z$ and one of its boundary components $S_3'$ is parallel to $S_3$. Using $\alpha$ is a guide, one can isotope $T_{S_1}\circ T_{S_2}$ to a diffeomorphism supported on a collar neighborhood of $S_3'$ and that is a sphere twist on this collar neighborhood. Since $S_3'$ is parallel to $S_3$, this gives the desired relation.

\begin{proof}[Proof of Theorem \ref{thm:twist-group}] We prove the theorem in 3 steps. 

\p{Step 1: a generating set for \boldmath $\Twist(M)$} For $i=1,\ldots,\ell$, fix an embedded sphere $S_i\subset Y$ that is parallel to the boundary component of $P_i\setminus\text{int}(D_i)$. For $j=1,\ldots,k$, let $S_j'\subset Y$ be the embedded sphere $S^2\times\{0\}$ in the $j$-th copy of $S^2\times[0,1]$. Let $\Twist(Y)$ be the subgroup of $\Mod(Y)$ generated by the sphere twists $\{\tau_{S_1},\ldots,\tau_{S_\ell}\}\cup\{\tau_{S_1'},\ldots,\tau_{S_k'}\}$. (Recall that the mapping class group $\Mod(Y)$ of a manifold with boundary is the group $\Diff_\partial(Y)$ of diffeomorphisms that restrict to the identity on $\partial Y$, modulo isotopies that are the identity on $\partial Y$.) 
Consider the composition 
\begin{equation}\label{eqn:twist-composition}(\Z/2\Z)^{\ell+k}\xrightarrow{\rho}\Twist(Y)\xrightarrow{\pi}\Twist(M),\end{equation}
where 
\[\rho(a_1,\ldots,a_\ell,b_1,\ldots,b_k)=\tau_{S_1}^{a_1}\>\cdots\>\tau_{S_\ell}^{a_\ell}\>\tau_{S_1'}^{b_1}\>\cdots\>\tau_{S_k'}^{b_k},\] 
and $\pi$ is the restriction of the homomorphism $\Mod(Y)\to\Mod(M)$. The composition $\pi\circ\rho$ is surjective by \cite[Prop.\ 1.2]{mccullough}. In the rest of the proof we compute the kernels of $\pi$ and $\rho$. 

\p{Step 2: global relation among sphere twists} In this step we compute the kernel of $\pi:\Twist(Y)\to\Twist(M)$. 

Let $D\subset X\subset M$ be an embedded ball. Let $\Emb(D,M)$ be the space of embeddings that respect the orientation. Let $\Emb_D^e(X,M)$ be the space of embeddings $X\to M$ that (i) restrict to the inclusion on $D$ and (ii) that extend to a diffeomorphism of $M$. Consider the following diagram, which consists of two fiber sequences (c.f.\ \cite{palais}). 

\[
\begin{xy}
(-30,0)*+{\Diff_\partial(Y)}="A";
(0,0)*+{\Diff(M,D)}="B";
(30,0)*+{\Emb_{D}^e(X,M)}="C";
(0,-15)*+{\Diff^+(M)}="E";
(0,-30)*+{\Emb(D,M)}="H";
{\ar"A";"B"}?*!/_3mm/{};
{\ar "B";"C"}?*!/_3mm/{};
{\ar "B";"E"}?*!/_3mm/{};
{\ar "E";"H"}?*!/_3mm/{};
\end{xy}\]

The ``horizontal" fiber bundle in the diagram splits as a product by \cite[Thm.\ 1]{hendriks-mccullough}. Consequently, $\Mod(Y)\to\Mod(M,D)$ is injective. Then from the proceeding diagram, we obtain 
\[\begin{xy}
(-30,-15)*+{\Mod(Y)}="A";
(0,0)*+{\pi_1\big(\Emb(D,M)\big)}="B";
(0,-15)*+{\Mod(M,D)}="E";
(0,-30)*+{\Mod(M)}="H";
{\ar@{^{(}->}"A";"E"}?*!/_3mm/{};
{\ar "B";"E"}?*!/_3mm/{\delta};
{\ar "E";"H"}?*!/_3mm/{};
\end{xy}\]
As is well-known, the space $\Emb(D,M)$ is homotopy equivalent to the (oriented) frame bundle of $M$, which is diffeomorphic to $M\times\SO(3)$ since closed oriented 3-manifolds are parallelizable. Then $\pi_1\big(\Emb(D,M)\big)\cong\pi_1(M)\times\Z/2\Z$. A generator of the $\Z/2\Z$ factor maps under $\delta$ to a sphere twist about $S:=\partial D$. From Lemma \ref{lem:pants} we deduce that $\tau_S=\tau_{S_1}\cdots\tau_{S_\ell}$ in $\Mod(M,D)$, so $\tau_{S_1}\cdots\tau_{S_\ell}$ belongs to $\ker(\pi)$. 

We claim that $\ker(\pi)=\langle\tau_{S_1}\cdots\tau_{S_\ell}\rangle$. To see that $\ker(\pi)$ is not larger, it suffices to show that if $\gamma\in\pi_1(M)<\pi_1\big(\Emb(D,M)\big)$ is nontrivial, then $\delta(\gamma)$ is not in the image of $\Twist(Y)\to\Mod(M,D)$. This is easy to see because each element of $\Twist(Y)<\Mod(M,D)$ acts trivially on $\pi_1(M,*)$ (where the basepoint $*$ belongs to $D$), whereas $\delta(\gamma)$ acts by a nontrivial conjugation on $\pi_1(M,*)$ (note that the fundamental group of a reducible 3-manifold has trivial center). 

\p{Step 3: local triviality of sphere twists} 
Here we compute the kernel of the map $\rho:(\Z/2\Z)^{\ell+k}\to\Twist(Y)$ defined in (\ref{eqn:twist-composition}). Since the spheres $S_1,\ldots,S_\ell,S_1',\ldots,S_k'$ belong to distinct components of $Y$, it suffices to determine which of the given generators for $\Twist(Y)$ is trivial in $\Twist(Y)$. 
Sphere twists in $S^2\times[-1,1]$ components are nontrivial by \cite{laudenbach}. See also \cite{BBP} who prove this by considering the action on framings. It remains then to consider when the twists $\tau_{S_i}$ (about the boundary of $P_i\setminus\text{int}(D_i)$) are nontrivial. By  Theorem \ref{thm:HFW} below, $\tau_{S_i}\in\Mod(P_i,D_i)<\Mod(Y)$ is trivial if and only if $P_i$ is a lens space. Combining this with the proceeding steps finishes the proof. 
\end{proof}

\begin{theorem}[Hendriks, Friedman-Witt]\label{thm:HFW}
Let $P$ be an irreducible $3$-manifold. Fix an embedded ball $D\subset P$, and let $\tau_S\in\Mod(P,D)$ be the sphere twist about a sphere $S$ parallel to $\partial D$. Then $\tau_S=1$ if and only if $P$ is a lens space. 
\end{theorem}

\begin{proof}
By work of Hendriks \cite{hendriks} (see \cite[Cor.\ 2.1]{friedman-witt}), the twist $T_S\in\Diff_\partial(P\setminus \text{int}(D))$ is not homotopic to the identity (rel boundary), unless $P$ is either a lens space or a prism manifold (the latter are manifolds covered by $S^3$ whose fundamental group is an extension of a dihedral group). 

First consider the case when $P$ is a lens space. Since we assume $P$ is irreducible, $P\neq S^1\times S^2$, so $P$ is covered by $S^3$. In this case $T_S$ is isotopic to the identity (rel boundary) by \cite[Lem.\ 3.5]{friedman-witt}. (Aside: $T_S$ is also isotopic to the identity when $P=S^1\times S^2$, which can be seen using Lemma \ref{lem:pants}.)

When $P$ is a prism manifold, $T_S$ is not isotopic to the identity (rel boundary). See \cite[Thm.\ 2.2]{friedman-witt}. The statement there does not include one family of prism manifolds $S^3/D_{4m}^*$. This is because the argument uses the (generalized) Smale conjecture, which was not proved for $S^3/D_{4m}^*$ at the time the paper was written. See \cite[Remark after Corollary 2.2]{friedman-witt}. Fortunately, the generalized Smale conjecture has now been confirmed for all prism manifolds (in fact for all elliptic 3-manifold, with the exception of $\R P^3$). See \cite{smale-elliptic} and \cite{balmer-kleiner}.
\end{proof}

For later use, we record the following Corollary of Theorem \ref{thm:twist-group}. For the manifold $S^1\times S^2$, we call a sphere of the form $*\times S^2$ a \emph{belt sphere} (we use this terminology because this sphere can be viewed as the belt sphere of a handle attachment).

\begin{corollary}\label{cor:twist-generator}
Let $M=\#_k(S^1\times S^2)\#P_1\#\cdots\#P_\ell$, and assume each $P_i$ is a lens space. Then $\Twist(M)\cong(\Z/2\Z)^k$ is generated by twists about the belt spheres of the $S^1\times S^2$ summands. 
\end{corollary}

\section{Decomposing finite group actions on 3-manifolds}\label{sec:invariant-spheres}

In this section we explain some general structural results for certain finite group actions on $3$-manifolds, which will allow us to decompose a $G$-manifold $M^3$ into simpler $G$-invariant pieces. For our application to the Main Theorem we are particularly interested in actions that are trivial on $\pi_i(M)$ for $i=1,2$.

\subsection{Equivariant sphere theorem}
The main result of this section is Theorem \ref{thm:invariant-spheres}. In order to state it, we introduce some notation. Let $\mathbb S$ be a collection of disjoint embedded spheres in a 3-manifold $M$. Define $M_{\mathbb S}$ as the result of removing an open regular neighborhood of each $S\in\mathbb S$ and capping each boundary component with a 3-ball. The 3-manifold $M_{\mathbb S}$ is a closed, but usually not connected. This process is pictured in Figure \ref{fig:cut-cap}.

\begin{figure}[h!]
\labellist
\pinlabel $M$ at 60 10
\pinlabel $M_{\mathbb S}$ at 580 10
\endlabellist
\centering
\includegraphics[scale=.6]{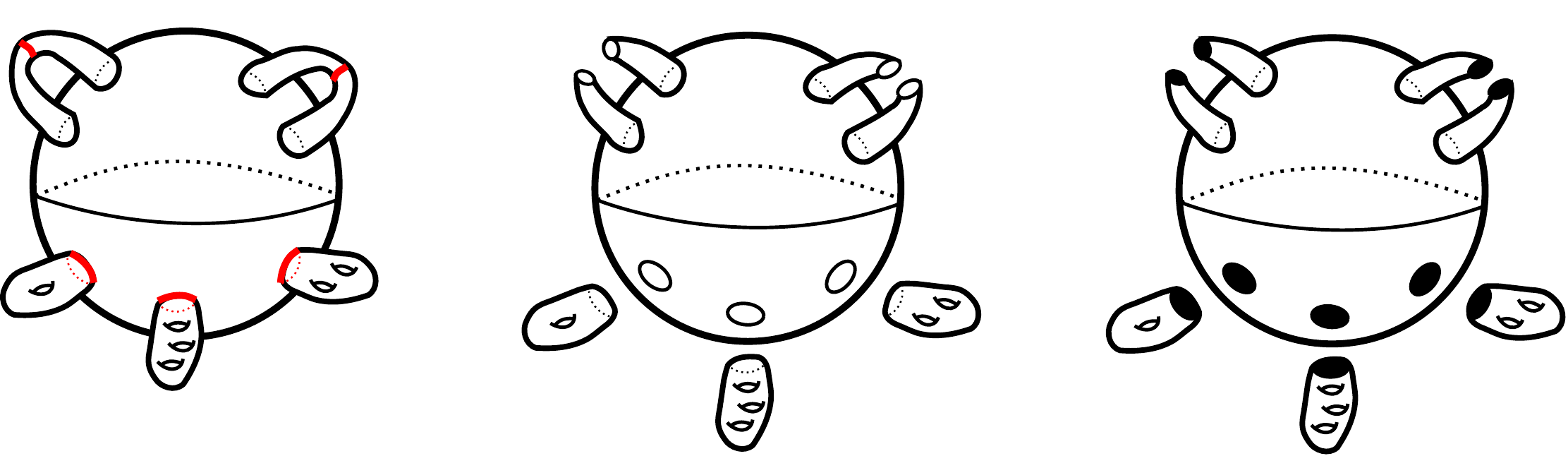}
\caption{Cutting and capping along spheres $M\leadsto M_{\mathbb S}$.}
\label{fig:cut-cap}
\end{figure}

\begin{theorem}\label{thm:invariant-spheres}
Let $M$ be a closed oriented $3$-manifold and let $G$ be a finite subgroup of $\Diff^+(M)$.
\begin{enumerate}
\item There exists a $G$-invariant collection $\mathbb S$ of disjoint embedded spheres in $M$ such that the components of $M_{\mathbb S}$ are irreducible.
\item If $M\neq  S^1\times S^2$ and $G$ acts trivially on $\pi_1(M)$ and $\pi_2(M)$, then $G$ preserves every element in $\mathbb S$. 
\end{enumerate} 
\end{theorem}

We call a collection of spheres as in the statement of Theorem \ref{thm:invariant-spheres} a \emph{sphere system} for $G$.

\begin{remark}\label{rmk:essential}
Without loss of generality, one can assume that no $S\in\mathbb S$ bounds a ball in $M$ by removing any sphere that bounds a ball from $\mathbb S$. Similarly, if $G$ preserves every element of $\mathbb S$, then we can also assume that no pair $S\neq S'\in\mathbb S$ bound an embedded $S^2\times[0,1]$ in $M$.
\end{remark}

Part (1) of Theorem \ref{thm:invariant-spheres} is due to Meeks--Yau; see \cite[c.f.\ Thm.\ 7]{MY}. An alternate approach was given by Dunwoody \cite[Thm.\ 4.1]{dunwoody}. The main tools used in these works are minimal surface theory (in the smooth and PL categories). For the proof of Theorem \ref{thm:invariant-spheres}(2), we use the following lemmas.

\begin{lemma}\label{lem:annulus}
Let $S_0,S_1\subset M$ be disjoint embedded spheres. If $S_0$ and $S_1$ are ambiently isotopic, then they bound an embedded $S^2\times[0,1]$ in $M$. 
\end{lemma}

Lemma \ref{lem:annulus} follows from \cite[Lem.\ 1.2]{laudenbach} and the Poincar\'e conjecture (Laudenbach proves that homotopic spheres bound an $h$-cobordism, and every $h$-cobordism is trivial by Perelman's resolution of the Poincar\'e conjecture).

\begin{lemma}\label{lem:annulus-homeo}
Let $h$ be an orientation-preserving homeomorphism of $S^2\times [0,1]$. If $h$ that interchanges the two boundary components, then $h$ acts on $H_2(S^2\times[0,1])\cong\Z$ by $-1$. 
\end{lemma}

\begin{proof}[Proof of Lemma \ref{lem:annulus-homeo}]
Set $A=S^2\times[0,1]$. Consider the arc $\alpha=*\times[0,1]$ and the sphere $\beta=S^2\times 0$. After orienting $\alpha$ and $\beta$, we view them as homology classes $\alpha\in H_1(A,\partial A)$ and $\beta\in H_2(A)$, which generate these groups. Since $h$ interchanges the components of $\partial A$, $h(\alpha)=-\alpha$. Since $h$ is orientation-preserving,  
\[\alpha\cdot\beta=h(\alpha)\cdot h(\beta)=-\alpha\cdot h(\beta).\]
This implies $h(\beta)=-\beta$ because the intersection pairing $H_1(A,\partial A)\times H_2(A)\rightarrow \Z$ is a perfect pairing by Poincar\'e--Lefschetz duality. 
\end{proof}

\begin{proof}[Proof of Theorem \ref{thm:invariant-spheres}(2)]

Let $\mathbb S$ be a $G$-invariant collection of embedded spheres as in Theorem \ref{thm:invariant-spheres}(1). Fix $S\in\mathbb S$ and $g\in G$. We want to show that $g(S)=S$. Suppose for a contradiction that $g(S)$ is disjoint from $S$. Fix an embedding $f:S^2\rightarrow M$ with $f(S^2)=S$. Since $g$ acts trivially on $\pi_2(M)$, the maps $f$ and $g\circ f$ are homotopic, hence isotopic by a result of Laudenbach and the Poincar\'e conjecture; c.f.\ \cite[Thm.\ 1]{laudenbach}.

By Lemma \ref{lem:annulus}, the spheres $S$ and $g(S)$ bound a submanifold $A\cong S^2\times[0,1]$. Let $k\ge2$ be the smallest power of $g$ so that $g^k(S)=S$. 

First assume $k\ge3$. Since we can find a $G$-invariant metric on $M$, the interiors of $A$ and $g(A)$ are either equal or disjoint. Since $k\ge3$, $g^2(S)\neq S$, so $g(A)\neq A$. Consequently, $A\cup g(A)$ is diffeomorphic to $S^2\times[0,1]$. Similarly, we conclude that $A\cup g(A)\cup\cdots\cup g^{k-1}(A)$ is diffeomorphic to $S^2\times S^1$ (this is the only $S^2$-bundle over $S^1$ with orientable total space). This contradicts the assumption that $M\neq S^2\times S^1$. 

Now assume that $k=2$. As in the preceding paragraph, if $A$ and $g(A)$ have disjoint interiors, then $M=S^2\times S^1$. Therefore, $g(A)=A$. By Lemma \ref{lem:annulus-homeo}, $[g(S)]=-[S]$ in $H_2(M)$. By assumption that $G$ acts trivially on $\pi_2(M)$, we know $[g(S)]=[S]$. Therefore, $2[S]=0$ in $H_2(M)$. If $S\subset M$ is non-separating, then there is a closed curve $\gamma\subset M$ so that $[S]\cdot[\gamma]=1$; this implies that $[S]$ has infinite order in $H_2(M)$, which is a contradiction. If $S\subset M$ is separating, then $M\setminus A$ is a union of two components $M_1\sqcup M_2$ that are interchanged by $g$. This contradicts the fact that $G$ acts trivially on $\pi_1(M)\cong\pi_1(M_1)*\pi_1(M_2)$. 
\end{proof}

\subsection{Decomposing an action along invariant spheres}\label{sec:decompose}

Here we explain how we use Theorem \ref{thm:invariant-spheres} to decompose an action $G\curvearrowright M$ into smaller pieces. We also prove a result about the action on the fundamental group of the pieces under the assumption that $G$ acts trivially on $\pi_1(M)$. 

Fix a finite subgroup $G<\Diff^+(M)$ and assume $G$ acts trivially on $\pi_1(M)$. Let $\mathbb S$ be a sphere system for $G$ (Theorem \ref{thm:invariant-spheres}). We will always assume that $\mathbb S$ has the additional properties discussed in Remark \ref{rmk:essential}: no $S\in\mathbb S$ bounds a ball and no pair $S,S'\in\mathbb S$ bound an embedded $S^2\times[0,1]$. 

Observe that there is an induced action of $G$ on $M_{\mathbb S}$. To construct it, recall a classical result of Brouwer, Eilenberg, and de Ker\'ekj\'art\'o \cite{Brouwer,K,Eilenberg} that every finite subgroup of $\Homeo^+(S^2)$ is conjugate to a finite subgroup of $\SO(3)$, hence extends from the unit sphere $S^2\subset\R^3$ to the unit ball $D^3\subset\R^3$. In this way the action of $G$ on $M\setminus\bigcup_{S\in\mathbb S}S$ extends to an action on $M_{\mathbb S}$, which can be made smooth as well.

\begin{remark}[global fixed points]\label{rmk:fixed-points} Since $G$ acts trivially on $\mathbb S$, then the center of each of the added 3-balls in $M_{\mathbb S}$ contains a global fixed point for the $G$-action; we call these \emph{canonical fixed points}. Each component of $M_{\mathbb S}$ contains at least one canonical fixed point.  \end{remark}

The following proposition will be important for our proof of the Main Theorem.

\begin{proposition}\label{prop:fundamental-group}
Fix $G<\Diff^+(M)$ acting trivially on $\pi_1(M)$ and fix a $G$-invariant collection $\mathbb S$ of disjoint, embedded spheres such that $M_{\mathbb S}$ has irreducible components. Let $N$ be a component of $M_{\mathbb S}$, and let $p\in N$ be a canonical fixed point, as defined in Remark \ref{rmk:fixed-points}. Then $G$ acts trivially on $\pi_1(N,p)$. 
\end{proposition}

\begin{proof}
Fix $g\in G$. We show that the action of $g$ on $\pi_1(N,p)$ is trivial. The statement is only interesting when $N$ is not simply connected, so we assume this. 

Let $k$ be the number of elements of $\mathbb S$ that meet $N$. We separate the argument into the cases $k=1$, $k=2$, and $k\ge3$. 

\p{Case: $k=1$} Let $S\in\mathbb S$ be the sphere that meets $N$. The sphere $S$ is separating and gives a description of $M$ as a connected sum $M\cong N\# N'$. Since $g\in G$ has finite order and preserves $S$, there is a fixed point $q\in S^g$. By van Kampen's theorem, $\pi_1(M,q)\cong\pi_1(N,q)*\pi_1(N',q)$. By Lemma \ref{lem:twist-action}, the action of $g$ on $\pi_1(M,q)$ is by conjugation by some element $\pi_1(M,q)$. Since $g$ preserves the decomposition $M=N\#N'$, it also preserves the factors in the splitting $\pi_1(N,q)*\pi_1(N',q)$. Note that  $\pi_1(N',q)$ is nontrivial, since otherwise $S$ bounds a ball in $M$, contrary to our assumption. The only conjugation of $A*A'$ that preserves both $A\neq1$ and $A'\neq1$ is the trivial conjugation, so $g$ acts trivially on $\pi_1(M,q)$. Consequently, $g$ acts trivially on $\pi_1(N,q)$, and also on $\pi_1(N,p)$. (In general, changing the basepoint can change the automorphism to a nontrivial conjugation, but in this does not happen here since e.g.\ the points $p,q\in N$ are connected by an arc contained in the fixed set $N^g$.) 

\p{Case: $k=2$} Let $S,S'\in\mathbb S$ denote the spheres the meet $N$. Observe that these sphere are either both separating or both nonseparating. Let $N'$ be the closed 3-manifold such that $M_{\mathbb S}=N\sqcup N'$. 

If $S$ and $S'$ are both separating, then the argument is similar to the case $k=1$. Apply that argument to either sphere to see that the action is trivial at the corresponding canonical fixed point. 


Assume then that both $S$ and $S'$ are nonseparating; in particular, this implies that $N'$ is connected. Then $M$ is obtained from $N\sqcup N'$ by removing balls $B_1,B_2\subset N$ and $B_1',B_2'\subset N'$ and gluing $\partial B_i$ to $\partial B_i'$. Choose a fixed points $q\in S^g$ and $q'\in (S')^g$. There is an isomorphism $\pi_1(M)\cong \pi_1(N,q)*\pi_1(N',q)*\Z$. (The $\Z$ factor is not important for this part of the argument, but will play a role when $k\ge3$.) 

By assumption $g$ acts on $\pi_1(M,q)$ by conjugation and preserves the free factors $\pi_1(N,q)$ and $\pi_1(N',q)$. Both of these groups is nontrivial, by our assumption that no two spheres in $\mathbb S$ are parallel. Then as before, we conclude that $g$ acts trivially on $\pi_1(M,q)$, hence also on $\pi_1(N,q)$ and $\pi_1(N,p)$. 

\p{Case $k\ge3$} Let $S_0,\ldots,S_{k-1}\in\mathbb S$ denote the spheres that meet $N$. If some $S_i$ separates, then we can proceed similar to the case $k=1$, so we can assume each $S_i$ is nonseparating. Then $M_{\mathbb S}=N\sqcup N'$, where $N'$ is connected, and $M$ is obtained from $N\sqcup N'$ by removing balls $B_0,\ldots,B_{k-1}\subset N$ and $B_0',\ldots,B_{k-1}'\subset N'$ and gluing $B_i$ and $B_i'$ along their boundary (which is identified with $S_i\subset M$). Choose fixed points $q_i\in (S_i)^g$. There is an isomorphism
\[\pi_1(M,q_0)\cong\pi_1(N,q_0)*\pi_1(N',q_0)*F_{k-1}.\]
If $\pi_1(N',q_0)\neq1$, then we can argue similar to the case $k=2$. Therefore, we assume that $N'$ is simply connected, which means $\pi_1(M,q_0)\cong\pi_1(N,q_0)*F_{k-1}$. 

The free group $F_{k-1}$ is generated by loops $\gamma_i=\eta_i*\eta_i'$, where $\eta_i$ is a path in $N$ from $q_0$ to $q_i$ (and disjoint from the interiors of the balls $B_0,\ldots,B_{k-1}$), and $\eta_i'$ is a path in $N'$ from $q_i$ to $q_0$.

On the one hand,\footnote{Here the symbol $\sim$ indicates homotopic loops based at $q_0$. For the first homotopy, note that the paths $g(\eta_i')$ and $\eta_i'$ are homotopic rel endpoints because $N'$ is simply connected.} 
\[g(\gamma_i)\sim g(\eta_i)*\eta_i'\sim g(\eta_i)*\overline{\eta_i}*\eta_i*\eta_i'=(g(\eta_i)*\overline{\eta_i})*\gamma_i,\]
so $g$ acts on $\gamma_i$ by left multiplication by the element $\beta_i = g(\eta_i)*\overline{\eta_i}\in\pi_1(N,q_0)$. 

On the other hand, $g$ acts on $\gamma_i$ by conjugation by an element $\alpha\in \pi_1(M,q_0)$. The only way these actions are equal is if both are trivial.  If we use the word length on $\pi_1(M,q_0)$ given by the generating set $\{s: s\in \pi_1(N,q_0) \text{ or } s\in F_{k-1}\}$, then the word length of $\alpha\gamma_i\alpha^{-1}$ is odd, but the word length for $\beta_i\gamma_i$ is 2 unless  $\beta_i=1$. This implies that $\beta_i=1$. Then we know that  $\gamma_i= \alpha\gamma_i\alpha^{-1}$ for every $i$, which implies that $\alpha=1$.

In particular, this implies that $g$ acts trivially on $\pi_1(M,q_0)$ and hence also on $\pi_1(N,q_0)$ and $\pi_1(N,p)$. 
\end{proof}

\section{Obstructing realizations}\label{sec:obstruction}

In this section we prove the ``only if" direction of the Main Theorem. This can be deduced quickly from the following more general statements. 

\begin{theorem}\label{thm:holder-argument}
Let $N$ be a closed, oriented, irreducible $3$-manifold with basepoint $p\in N$. Suppose there exists a nontrivial, finite-order element $f\in\Diff^+(N,p)$ that acts trivially on $\pi_1(N,p)$. Then $N$ is a lens space.  
\end{theorem}

\begin{theorem}\label{thm:noncyclic}
Let $M$ be a closed, oriented, reducible $3$-manifold. Let $G<\Diff^+(M)$ be a finite subgroup that acts trivially on $\pi_1(M)$. Then $G$ is cyclic. 
\end{theorem}

\begin{proof}[Proof of Main Theorem: obstruction]
Suppose $1\neq G<\Twist(M)$ is realizable. The fact that $\Twist(M)\neq1$ implies that either $M=S^2\times S^1$ or $M$ is reducible. In the former case, there is nothing to prove, so we assume $M$ is reducible. This assumption together with Lemma \ref{lem:twist-action} allow us to apply Theorem \ref{thm:noncyclic} and conclude that $G$ is cyclic. 

It remains to show $M$ is a connected sum of lens spaces, or, equivalently, that each component of $M_{\mathbb S}$ is a lens space. This is implied directly by Proposition \ref{prop:fundamental-group} and Theorem \ref{thm:holder-argument}. 
\end{proof}

Next we use Theorem \ref{thm:holder-argument} to deduce Theorem \ref{thm:noncyclic}. Then we prove Theorem \ref{thm:holder-argument}. 

\begin{proof}[Proof of Theorem \ref{thm:noncyclic}]
Let $\mathbb S$ be a sphere system for $G$ (Theorem \ref{thm:invariant-spheres}). We also assume that no $S\in\mathbb S$ bounds a ball and that no two spheres $S,S'\in\mathbb S$ bound an embedded $S^2\times[0,1]$ (Remark \ref{rmk:essential}). 

Since $G$ acts trivially on $\pi_1(M)$, Proposition \ref{prop:fundamental-group} and Theorem \ref{thm:holder-argument} combine to show that each component of $M_{\mathbb S}$ is a lens space. 

Suppose that there exists a component $N$ of $M_{\mathbb S}$ that is diffeomorphic to $S^3$. Let $k$ be the number of elements of $\mathbb S$ that meet $N$. If $k=1$, then $N$ is obtained from $M$ by cutting along a sphere that bounds a ball. Similarly, if $k=2$, then $N$ is obtained by cutting $M$ along two parallel spheres. Both of these are contrary to our assumption about $\mathbb S$. Therefore $k\ge3$, which implies that $N^G$ has at least 3 points. By the Smith conjecture \cite{smith-conjecture}, the action of $G$ on $N\cong S^3$ is conjugate into $\SO(4)$, and the fact that $|N^G|\ge3$ implies that $G$ is conjugate into $\SO(2)$. Therefore $G$ is cyclic. 

The remaining case is that every component of $M_{\mathbb S}$ is a lens space different from $S^3$ and $S^1\times S^2$. In this case $\Twist(M)$ is the trivial group by Corollary \ref{cor:twist-generator}. 
\end{proof}

We proceed to the proof of Theorem \ref{thm:holder-argument}. Our argument is inspired by an argument of Borel \cite{borel-isometry-aspherical} that shows that a finite group $G$ acting faithfully on a closed aspherical manifold $N$ and $\pi_1(N)$ has trivial center, then $G$ also acts faithfully on $\pi_1(N)$ (by outer automorphisms). 

Before starting the proof, we recall some facts about lifting actions to universal covers. 
Let $N$ be a closed manifold. Recall that $\widetilde N$ can be defined as the set of paths $\alpha:[0,1]\to N$ with $\alpha(0)=*$, up to homotopy rel endpoints. Using this description, there is a left action $\pi_1(N,*)\times \widetilde N\to \widetilde N$ given by pre-concatenation of paths $[\gamma].[\alpha]=[\gamma*\alpha]$, and there is a left action 
\[\Diff(N,*)\times\widetilde N \to\widetilde N\] given by post-composition $f.[\alpha]=[f\circ \alpha]$. If $f\in\Diff(N,*)$ acts trivially on $\pi_1(N,*)$, then $F([\alpha])=[f\circ \alpha]$ is a lift of $f$ that commutes with the deck group action and fixes the homotopy class of the constant path (as well as every other homotopy class corresponding to an element of $\pi_1(N,*)$). 

\begin{proof}[Proof of Theorem \ref{thm:holder-argument}]
As observed above, we can lift $f$ to a finite-order diffeomorphism $F$ that commutes with the deck group $\pi_1(N,*)$ and has a global fixed point. 

First we show that $\pi_1(N)$ is finite. Suppose for a contradiction that $\pi_1(N)$ is infinite. This implies $\widetilde N$ is contractible.\footnote{By Hurewicz, $\pi_3(\widetilde N)\cong H_3(\widetilde N)$. Since $\pi_1(N)$ is infinite, $\widetilde N$ is noncompact, so $H_3(\widetilde N)=0$. Similarly, all higher homotopy groups vanish by Hurewicz's theorem.} By Smith theory \cite{Smith}, the fixed set $(\widetilde{N})^F$ is connected, and simply connected. Since $F$ acts smoothly, $(\widetilde{N})^F$ is a smooth 1-dimensional manifold, hence it is homeomorphic to $\R$. Since $\pi_1(N)$ commutes with $F$, it acts on $(\widetilde N)^F\cong\R$, and this action is free and properly discontinuous since the action of $\pi_1(N,*)$ on $\widetilde{N}$ has these properties. This implies that $\pi_1(N,*)\cong\Z$, which contradicts the fact that $N$ is a closed, aspherical 3-manifold ($\Z$ is not a 3-dimensional Poincar\'e duality group). 

Since $\pi_1(N)$ is finite, its universal cover is diffeomorphic to $S^3$ by the Poincare-conjecture. As in the preceding paragraph, consider the action of $F$ on $\widetilde N\cong S^3$. By Smith theory and smoothness of the action, the fixed set is a smooth, connected 1-dimensional manifold with nontrivial fundamental group. Hence $(\widetilde N)^F\cong S^1$. Since $\pi_1(N)$ acts freely on $(\widetilde N)^F$ this implies $\pi_1(N)$ is cyclic, which implies that $N$ is a lens space. 
\end{proof}

\section{Constructing realizations}\label{sec:construction}

In this section we prove the ``if" direction of the Main Theorem. We state this as the following theorem. 

\begin{theorem}\label{thm:realization}
Let $M$ be a connected sum of lens spaces. Then every cyclic subgroup of $\Twist(M)$ is realizable. 
\end{theorem}

Fix $M$ as in Theorem \ref{thm:realization}, and write the prime decomposition
\[M=\#_k(S^1\times S^2)\# P_1\#\cdots\# P_\ell,\]
where each $P_i$ is a lens space different from $L(0,1)\cong S^1\times S^2$. 

To prove Theorem \ref{thm:realization}, given a nontrivial element $g\in\Twist(M)$ we define $\gamma\in\Diff(M)$ such that $\gamma^2=id$ and $[\gamma]=g$ in $\Mod(M)$. The basic approach is to define an order-2 diffeomorphism of 
\begin{equation}\label{eqn:disjoint-union}\sqcup_k(S^1\times S^2)\sqcup P_1\sqcup\cdots\sqcup P_\ell\end{equation} in such a way that the diffeomorphisms on the components can be glued to give an order-2 diffeomorphism of $M$. On each component of (\ref{eqn:disjoint-union}) we perform one of the following diffeomorphisms. 
\begin{itemize}
\item (constant $\pi$ rotation) Define 
\[R_0:S^1\times S^2\to S^1\times S^2\] by $id\times r$, where $r:S^2\to S^2$ is any $\pi$ rotation (choose one -- the particular axis is not important). 
\item (nonconstant $\pi$ rotation) Let $c:[0,1]\to\R P^2$ be a closed path that generates $\pi_1(\R P^2)$, and let $\alpha:\R P^2\rightarrow\SO(3)$ be the map that sends $\ell\in\R P^2$ to the $\pi$-rotation whose axis is $\ell$. Now define 
\[R_1:S^1\times S^2\to S^1\times S^2\] by $(t,x)\mapsto(t,\alpha(c(t))(x))$. 

Since $\alpha\circ c:[0,1]\rightarrow\SO(3)$ defines a nontrivial element of $\pi_1(\SO(3))$, the diffeomorphism $R_1$ represents the generator of $\Twist(S^1\times S^2)\cong\Z/2\Z$. This shows that $\Twist(S^1\times S^2)$ is realized. This involution appears in \cite[\S1]{tollefson_involutions} in a slightly different form.
\item (lens space rotation) Fix $p,q$ relatively prime and with $p\ge2$. View $L(p,q)$ as the quotient of $S^3\subset\mathbb C^2$ by the $\Z/p\Z$ action generated by $(z,w)\mapsto(e^{2\pi i/p}z, e^{2\pi iq/p}w)$. Define
\[R_{p,q}:L(p,q)\to L(p,q)\]
as the involution induced by $(z,w)\mapsto(z,-w)$ on $S^3$ (which descends to $L(p,q)$ since it commutes with the $\Z/p\Z$ action).
\end{itemize}

Each of the diffeomorphisms $R_0$, $R_1$, and $R_{p,q}$ has 1-dimensional fixed set. The representation in the normal direction at a fixed point is the antipodal map on $\R^2$ (there is no other option since these diffeomorphisms are involutions). Lemma \ref{lem:gluing} below allows us to glue these actions along their fixed sets. 

\begin{lemma}\label{lem:gluing}
Suppose $M,M'$ are oriented manifolds, each with a smooth action of a finite group $G$. Assume that $x\in M$ and $x'\in M'$ are fixed points of $G$, and that the representations $T_xM$ and $T_{x'}M'$ are isomorphic by an orientation reversing map. Then $M$ and $M'$ can be glued along regular neighborhoods $B$ and $B'$ of $x$ and $x'$ so that there is a smooth action of $G$ on $M\#M'$ that restricts to the given action on $M\setminus B$ and $M'\setminus B'$. \qed
\end{lemma}

\begin{remark}
The condition that the isomorphism $T_xM\cong T_{x'}M'$ be orientation-reversing appears because the connected sum of two oriented manifolds is defined by deleting an open ball from each and identifying the boundaries of these balls by an orientation-reversing diffeomorphism. This condition is always satisfied if each tangent space contains a copy of the trivial representation (choose an appropriate reflection). 
\end{remark}

\begin{remark}[Useful isotopies]\label{rmk:isotopy}
To prove that $\gamma\in\Diff(M)$ is in the isotopy class of $g\in\Twist(M)$, the following observation will be useful. The fixed set of $R_{p,q}$ acting on $L(p,q)$ contains\footnote{It's possible that the fixed set is larger (this is true for $L(2,1)\cong\R P^3$), but this is not important.} the image $C$ of the circle $\{(z,0): |z|=1\}\subset S^3$. The isotopy $h_t(z,w)=(z,e^{\pi i(1-t)}w)$, $0\le t\le 1$, descends to $L(p,q)$ to give an isotopy between $R_{p,q}$ and the identity, and $h_t$ fixes $C$ for each $t$. 

Similarly, it's possible to isotope $R_1$ to $R_1'$, which is a constant $\pi$-rotation (say about the $z$-axis) on a neighborhood of $*\times S^2$, for some fixed basepoint $*\in S^1$ (observe that $R_1'$ is still an involution). Furthermore, we can isotope $R_1'$ to a diffeomorphism that is the identity near $*\times S^2$ and in such a way that the isotopy at time $t\in[0,1]$ is a rotation by angle $\pi (1-t)$ (about the $z$-axis) on each sphere in a regular neighborhood of $*\times S^2$. The fixed set restricted to a neighborhood of $*\times S^2$ remains constant during this isotopy. 

Finally, we can isotope $R_0$ to the identity so that at time $t$ the diffeomorphism is a constant rotation by angle $\pi(1-t)$ (about the fixed axis). 

On a neighborhood of a fixed point, the local picture of the isotopies of $R_{p,q}$, $R_1'$, and $R_0$ looks the same. This will allow us to perform these isotopies equivariantly on connected sums. 
\end{remark}

We proceed now to the proof of Theorem \ref{thm:realization}. First we warm up with the case $M=\#_k(S^1\times S^2)$ and then we do the general case. 

\subsection{Realizations for connected sums of \boldmath $S^1\times S^2$}\label{sec:torus-case}

Fix $k\ge1$ and consider 
\[M_k:=\#_k(S^1\times S^2).\] Let $S_i$ be a belt sphere in the $i$-th connect summand, and denote the sphere twist about $S_i$ by $\tau_i$. The twists $\tau_1,\ldots,\tau_k$ form a basis for $\Twist(M_k)\cong(\Z/2\Z)^k$, c.f.\ Theorem \ref{thm:twist-group}. 

Fix a nonzero element
\[g=a_1\tau_1+\cdots+a_k\tau_k\] in $\Twist(M_k)$. We start by defining an involution $\hat\gamma$ of $\sqcup_k(S^1\times S^2)$. For ease of exposition, let $W_i=S^1\times S^2$ denote the $i$-th component of $\sqcup_k(S^1\times S^2)$. Define $\hat\gamma$ on $W_i$ to be $R_0$ or $R_1$, depending on whether the coefficient $a_i$ is $0$ or $1$, respectively. Next we glue using Lemma \ref{lem:gluing} to obtain an involution $\gamma$ of $M_k=W_1\#\cdots\#W_k\cong\#_k(S^1\times S^2)$. There are multiple ways to describe the gluing; for example, choose $k-1$ distinct fixed points $x_1,\ldots,x_{k-1}\in W_k$, and for $1\le i\le k-1$, glue $W_i$ to $W_k$ along regular (equivariant) neighborhoods of $x_i$ and an arbitrary fixed point $y_i\in W_i$ (the neighborhoods of $x_1,\ldots,x_k$ should be chosen to be small enough so that they are disjoint).

To see that $\gamma\in\Diff(M_k)$ is in the isotopy class $g$, recall the short exact sequence of Laudenbach
\[1\to\Twist(M_k)\to\Mod(M_k)\to\Out(\pi_1(M_k))\to1.\]
It's easy to check that $\gamma$ acts trivially on $\pi_1(M_k)$, so $\gamma$ represents a mapping class in $\Twist(M_k)$. The particular isotopy class is determined by the action on trivializations of the tangent bundle of $M_k$, and in this way one can check that $[\gamma]=g$ in $\Twist(M_k)$. We do not spell out the details of this because we give an alternate argument in the next section in the general case. 

\begin{remark}
We cannot realize a non-cyclic subgroup of $\Twist(M_n)$ using this construction because it is not possible to choose the axis for $R_0$ so that (1) $R_0$ and $R_1$ have a common fixed point and (2) $R_0$ and $R_1$ commute. Indeed,  \S\ref{sec:obstruction} proves no non-cyclic subgroup of $\Twist(M_n)$ is realized. 
\end{remark}

\subsection{Realizations for connected sum of lens spaces} 

Now we treat the general case 
\[M=\#_k(S^1\times S^2)\# L(p_1,q_1)\#\cdots\# L(p_\ell,q_\ell),\] where each $L(p_j,q_j)$ is a lens space different from $L(0,1)\cong S^1\times S^2$. Our approach is similar to the preceding section. 

Recall from Corollary \ref{cor:twist-generator} that $\Twist(M)\cong(\Z/2\Z)^k$ is generated by twists $\tau_1,\ldots,\tau_k$ in the belt spheres of the $S^1\times S^2$ summands. 

Fix a nonzero element
\[g=a_1\tau_1+\cdots+a_k\tau_k\] in $\Twist(M)$. We start by defining an involution $\hat\gamma$ of 
\[\sqcup_k(S^1\times S^2)\sqcup L(p_1,q_1)\sqcup\cdots\sqcup L(p_\ell,q_\ell).\]
Let $W_i$ denote the $i$-th component diffeomorphic to $S^1\times S^2$. Define $\hat\gamma$ on $L(p_j,q_j)$ to be $R_{p_j,q_j}$, and on $W_i$ to be $R_0$ or $R_1'$, depending on whether the coefficient $a_i$ is $0$ or $1$, respectively. (Recall that $R_1'$ is similar to $R_1$, but it has a product region.) 

Next we glue using Lemma \ref{lem:gluing} to obtain an involution $\gamma$ of $M$. We glue by the following pattern. First we glue $W_1,\ldots,W_k$. Choose $k-1$ distinct fixed points $x_1,\ldots,x_{k-1}\in W_k$, and for $1\le i\le k-1$, glue $W_i$ to $W_k$ along regular (equivariant) neighborhoods of $x_i$ and an arbitrary fixed point $y_i\in W_i$ (as was done in the preceding section). Next glue $L(p_j,q_j)$ to  $W_k$ in a region where $\hat\gamma$ acts as a product (we can choose $x_1,\ldots,x_{k-1}$ and the regular neighborhoods of these points to ensure that there is room to do this). In this way we obtain an involution $\gamma\in\Diff(M)$. 

We need to check that $\gamma$ is in the isotopy class of $g$. Using the isotopies defined in Remark \ref{rmk:isotopy}, we can isotope $\hat \gamma$ to a map that is the identity on each $L(p_j,q_j)$ component and each component $W_i$ such that $a_i=0$, and is a sphere twist on each component $W_i$ such that $a_i=1$. By construction these isotopies glue to give an isotopy of $\gamma$ to a product of sphere twists representing $g$.

This completes the proof of Theorem \ref{thm:realization}.\qed

\begin{question}
Are any two realizations of $g\in\Twist(M)$ conjugate in $\Diff(M)$? 
\end{question}

\section{Burnside Problem for $\Diff(M)$ for reducible 3-manifold $M$}\label{Burn}
In this section, we prove Theorem \ref{Burnside}, which follows quickly from Lemma \ref{Burnside_lemma}. 
\begin{lemma}\label{Burnside_lemma}
Fix a closed, oriented $3$-manifold $M$, and consider the group 
\[K:=\ker\big[\Diff(M)\xrightarrow{\Phi} \Out(\pi_1(M))\big].\] If $M$ is reducible and that $M$ is not a connected sum of lens spaces, then $K$ is torsion free. 
\end{lemma}

\begin{remark}[A strong converse to Lemma \ref{Burnside_lemma}]
If $M$ is a connected sum of lens spaces, then $M$ has a faithful $S^1$-action, so $K$ contains $S^1$ as a subgroup. To see this, observe that each lens space has an $S^1$ action with global fixed points, so by performing the connected sum equivariantly along fixed points (similar to the construction in Section \ref{sec:construction}) we obtain an $S^1$ action on $M$.
\end{remark}


\begin{proof}[Proof of Theorem \ref{Burnside}]
If $\Diff(M)$ has an infinite torsion group $G$, then either the image of $\Phi$ is an infinite torsion group or the kernel $K$ of $\Phi$ has infinite torsion group. By \cite[Thm.\ 5.2]{Hong_McCullough}, $\Out(\pi_1(M))$ contains a torsion-free finite-index subgroup. This implies that $\Phi(G)$ is finite, since every element of $\Phi(G)$ has finite order. On the other hand, by Lemma  \ref{Burnside_lemma} that  $K$ is torsion free, so $G\cong \Phi(G)$ is finite.
\end{proof}

\begin{proof}[Proof of Lemma \ref{Burnside_lemma}]
Fix a nontrivial subgroup $G=\Z/p\Z<K$, and fix a sphere system $\mathbb S$ for $G$ (Theorem \ref{thm:invariant-spheres}). By Proposition \ref{prop:fundamental-group}, the action of $G$ on each component $N$ of $M_\mathbb{S}$ is trivial on $\pi_1(N,p)$ as an automorphism. This implies that each component $N$ of $M_\mathbb{S}$ is a lens space by Theorem \ref{thm:holder-argument}.
\end{proof}

\bibliographystyle{alpha}
\bibliography{citing}

\end{document}